\documentclass[a4paper]{amsart}

\usepackage{amsmath}
\usepackage{t1enc} 
\usepackage{epsf}
\usepackage{amssymb}
\usepackage{color}
\usepackage{graphicx}
\usepackage{bbm}

\newtheorem{theorem}{Theorem}[section]
\newtheorem{lem}[theorem]{Lemma}
\newtheorem{prop}[theorem]{Proposition}

\newtheorem{fact}[theorem]{Fact}

\theoremstyle{definition}
\newtheorem{defi}[theorem]{Definition}
\newtheorem{ex}[theorem]{Example}
\newtheorem{rem}[theorem]{Remark}
\newtheorem{prob}[theorem]{Problem}

\numberwithin{equation}{section}
\numberwithin{theorem}{section}

\newcommand{\LN}{\mathbbm{N}}
\newcommand{\LC}{\mathbbm{C}}
\newcommand{\LQ}{\mathbbm{Q}}

\newcommand{\LZ}{\mathbbm{Z}}

\newcommand{\kk}{\mathbbm{k}}
\newcommand{\zetan}{\zeta_{n}}

\newcommand{\Z}{\mathbbm{Z}}
\newcommand{\N}{\mathbbm{N}}
\newcommand{\Zn}{\mathbbm{Z}[\zeta_{n}]}

\newcommand{\Q}{\mathbbm{Q}}

\newcommand{\R}{\mathbbm{R}}
\newcommand{\C}{\mathbbm{C}}

\newdimen{\standardlabelwidth}
\standardlabelwidth-\standardlabelwidth
  \advance\standardlabelwidth by 2\normalparindent
\newcommand{\standardlabel}[1]{#1\kern\standardlabelwidth}

\begin{document}


\title[Uniqueness in discrete tomography of Delone sets]{Uniqueness in discrete tomography of\\ Delone sets with long-range order}
\author{Christian Huck}
\address{Department of Mathematics and Statistics\\
  The Open University\\ Walton Hall\\ Milton Keynes\\ MK7 6AA\\
   United Kingdom}
\email{c.huck@open.ac.uk}
\thanks{The author was supported by the German Research Council (Deutsche Forschungsgemeinschaft), within the CRC 701, and by EPSRC via Grant EP/D058465/1.}

\begin{abstract}
We address the problem of determining finite subsets of Delone sets
$\varLambda\subset\R^d$ with long-range order by $X$-rays in prescribed $\varLambda$-directions, i.e., directions parallel to
non-zero interpoint vectors of $\varLambda$. Here, an $X$-ray in direction
$u$ of a finite
 set gives the number of points in the set on each line
parallel to $u$. For our main result, we introduce the notion of algebraic Delone
sets $\varLambda\subset\R^2$ and derive a sufficient condition for the
determination of the convex subsets of these sets by $X$-rays in four prescribed
$\varLambda$-directions.
\end{abstract}

\maketitle

\section{Introduction}\label{intro}

{\em Discrete tomography} is concerned with the 
inverse problem of retrieving information about some {\em finite}
object from 
(generally noisy) information about its slices. An important problem  is the {\em reconstruction} of a finite point set
from its line sum functions in a small number of directions. More precisely, a ({\em discrete parallel}\/) {\em
  X-ray} of a finite subset of $\R^d$ in direction $u$ gives the
number of points in the set on each line parallel to
$u$. In the
 traditional setting, motivated by crystals, the positions to be determined form a subset of a common translate of a lattice in $\R^d$.  
In fact, many of the problems in discrete tomography
have been studied on the square lattice; see ~\cite{GG}, ~\cite{GG2}, ~\cite{GGP}, ~\cite{Gr} and~\cite{HK}. In the longer run, by
also having other structures than perfect crystals in mind, one has to
take into account wider classes of sets, or at least
significant deviations from the lattice structure. As an intermediate
step between periodic and random (or amorphous) \emph{Delone sets}, we consider Delone sets with {\em long-range order}, thus including systems of {\em aperiodic order} like \emph{model sets} (also called \emph{mathematical quasicrystals}) as a special case.

The main motivation for our interest in the discrete tomography of
Delone sets $\varLambda$ with long-range order comes from the fact that these sets serve as a rather general model of atomic positions in solid state materials together with the demand of materials
science to reconstruct such three-dimensional structures or planar
layers of them from their images under quantitative {\em high
  resolution transmission electron microscopy} (HRTEM). In fact, in~\cite{ks} and~\cite{sk} a technique
is described, which, for some
crystals, can effectively measure the number of atoms lying on densely occupied 
lines. Therefore, only $\varLambda$-directions, i.e.,
directions parallel to non-zero interpoint vectors of $\varLambda$,
will be considered. Further, since typical objects may be damaged or even destroyed by the radiation
energy after about $3$ to $5$ images taken by
HRTEM, one is restricted to a small 
number of $X$-rays. It actually is this restriction to {\em few high-density
directions} that makes the
problems of discrete tomography mathematically challenging, even if
one assumes the absence of noise. 

Since the above reconstruction problem of discrete tomography can possess rather
different solutions, one is led to the {\em
  determination} of finite subsets of a fixed Delone set
$\varLambda$ by $X$-rays in a
 small number of suitably prescribed $\varLambda$-directions. More
 precisely, we say that the elements of a collection $\mathcal{E}$ of
 finite sets are {\em determined} by the
$X$-rays in a finite set of directions if
different sets $F$ and $F'$ in $\mathcal{E}$ cannot have the same
$X$-rays in these directions. 

After summarizing a few general 
results on determination, we introduce the class
of {\em algebraic Delone sets}  
$\varLambda\subset\R^2$ (see Definition~\ref{algdeldef}) and study the determination of the {\em convex
  subsets} of these sets by $X$-rays in $\varLambda$-directions. They are finite subsets of $\varLambda$ with the property that their convex hull contains no new points of
$\varLambda$. By using
$p$-adic valuation methods as introduced in discrete tomography by
Gardner and Gritzmann in~\cite{GG} together with standard facts
from the theory of (cyclotomic) fields, we derive a sufficient condition
for the determination of the convex subsets of $\varLambda$ by
$X$-rays in four $\varLambda$-directions and show that three $\varLambda$-directions never
suffice for this purpose; cf.~Theorem~\ref{main3gen}. Further, by using standard
facts from algebraic number theory and the theory of
Pisot-Vijayaraghavan 
numbers, it is shown that \emph{cyclotomic model sets} (see Definition~\ref{deficy}) are examples of algebraic Delone sets; cf.~Proposition~\ref{cmsads}. We conclude
with a discussion of the result on the determination of convex
sets by four $X$-rays for this specific class of objects;
cf.~Theorem~\ref{main3gencoro}.

\section{Preliminaries and Notation}\label{found}

Natural numbers are always assumed to be positive and we denote by $\mathbbm{P}$ the
set of rational primes. We denote the norm in Euclidean $d$-space $\mathbbm{R}^{d}$
by $\Arrowvert \cdot \Arrowvert$. The unit sphere in $\mathbbm{R}^{d}$ is
denoted by $\mathbb{S}^{d-1}$ and its elements are also called
{\em directions}. For $r>0$ and $x\in\R^{d}$,
$B_{r}(x)$ is the open ball of radius $r$ about $x$. For a subset $S\subset
\mathbbm{R}^{d}$, a direction
$u\in\mathbb{S}^{d-1}$ is called an $S${\em-direction} if it is
parallel to a non-zero element of the difference set
$S-S$ of $S$. A {\em homothety} $h\!:\, \mathbbm{R}^{d} \rightarrow
\mathbbm{R}^{d}$ is given by $z \mapsto \lambda z + t$, where
$\lambda \in \R$ is positive and $t\in \mathbbm{R}^{d}$. A {\em convex polygon} is the convex hull of a finite set of points in $\R^2$. For a subset $S \subset \R^2$, a {\em polygon in} $S$ is a convex polygon with all vertices in $S$. Further, a bounded subset $C$ of $S$ is called a {\em convex subset of} $S$ if its
convex hull contains no new points of $S$, i.e., if $C =
\operatorname{conv}(C)\cap S$. Let $U\subset \mathbb{S}^{1}$ be
a finite set of directions. A non-degenerate convex polygon $P$ is
called a {\em $U$-polygon} if it has the property that whenever $v$ is a vertex of $P$ and $u\in U$, the line in the plane in direction $u$ which passes through $v$ also meets another vertex $v'$ of $P$. Occasionally, we 
identify $\C$ with $\R^{2}$ and $\bar{z}$ will always denote the
complex conjugate of $z$. Consider a set $\varLambda\subset\R^d$, where $d\in\N$. $\varLambda$ is called {\em uniformly discrete} if there is a radius
$r>0$ such that every ball $B_{r}(x)$ with $x\in\mathbbm{R}^{d}$ contains at most one point of
  $\varLambda$. Note that the bounded 
  subsets of a uniformly discrete set $\varLambda$ are precisely the
  finite subsets of $\varLambda$. $\varLambda$ is called {\em relatively dense} if there is a radius $R>0$
  such that every ball $B_{R}(x)$ with 
  $x\in\mathbbm{R}^{d}$ contains at least one point of
  $\varLambda$. $\varLambda$ is called a {\em Delone set} if it is both uniformly
  discrete and relatively dense. $\varLambda$ is said to be of \emph{finite local complexity} if $\varLambda-\varLambda$ is discrete and closed. Note that a subset 
  $\varLambda$ has finite local complexity if and
  only if for every $r>0$ there
  are, up to translation, only finitely many \emph{patches of radius
    $r$}, i.e., sets of the form $\varLambda\cap
  B_{r}(x)$, where $x\in\mathbbm{R}^d$; cf.~\cite{Moody}.  A Delone set $\varLambda$ is a \emph{Meyer set} if $\varLambda-\varLambda$ is
  uniformly discrete. Translates $\varLambda$ of arbitrary lattices
  $L\subset\R^d$, are simple examples of Meyer sets, since
  $\varLambda-\varLambda=L$ is a Delone set. Trivially, any Meyer set
  is of finite local complexity. Finally, $\varLambda$ is called {\em aperiodic} if it has no
non-zero translation symmetries.

\begin{defi}\label{xray..}
Let $F$ be a finite subset of $\mathbbm{R}^{d}$. Furthermore, let $u\in
\mathbb{S}^{d-1}$ be a direction and let $\mathcal{L}_{u}^{d}$ be the set
of lines in direction $u$ in $\mathbbm{R}^{d}$. Then the ({\em
  discrete parallel}\/) {\em X-ray} of $F$ {\em in direction} $u$ is
the function $X_{u}F: \mathcal{L}_{u}^{d} \rightarrow
\mathbbm{N}_{0}:=\mathbbm{N} \cup\{0\}$, defined by $$X_{u}F(\ell) :=
\operatorname{card}(F \cap \ell\,) =\sum_{x\in \ell}
\mathbbm{1}_{F}(x)\,.$$ Moreover, the {\em support} $(X_{u}F)^{-1}(\N)$ of $X_{u}F$, i.e.,
the set of lines in $\mathcal{L}_{u}^{d}$ which pass through at least one
point of $F$, is denoted by $\operatorname{supp}(X_{u}F)$. Further, for a finite set
  $U\subset\mathbb{S}^{d-1}$ of directions, set
$$
G^{F}_{U}\,\,:=\,\,\bigcap_{u\in U}\,\,\Big( \bigcup_{\ell \in
  \mathrm{supp}(X_{u}F)} \ell\Big)\,.
$$

\end{defi}

Note that, in the situation of Definition~\ref{xray..}, one has
$F\subset G^{F}_{U}$. Further, $G^{F}_{U}$ is finite if
$\operatorname{card}U\geq 2$.




\begin{fact}\label{homotu}
Let $h\!:\,\mathbbm{R}^{d}  \rightarrow \mathbbm{R}^{d}$ be a homothety, and let $U\subset \mathbb{S}^{d-1}$ be a finite set of directions. Then, one has:
\begin{itemize}
\item[(a)]
If $P$ is a $U$-polygon, then $h(P)$ is again a $U$-polygon.
\item[(b)]
If $F$ and $F'$ are finite subsets of $\mathbbm{R}^d$
with the same $X$-rays in the directions of $U$, then the finite sets $h(F)$ and $h(F')$ also have the same $X$-rays in the directions of $U$.
\end{itemize}
\end{fact}

\begin{defi}
Let $\mathcal{E}$ be a collection of finite subsets of
$\mathbbm{R}^{d}$ and let $U\subset\mathbb{S}^{d-1}$ be a finite set
of directions. We say that the elements
of $\mathcal{E}$ are {\em determined} by the $X$-rays in the directions of $U$ if, for all $F,F' \in \mathcal{E}$, one has
$$
(X_{u}F=X_{u}F'\;\,\forall u \in U) \;  \Longrightarrow\; F=F'\,.
$$
\end{defi}

\section{General results on determination}\label{gen}

We need the following property of sets
$\varLambda\subset\R^d$, where $\langle S\rangle_{\Z} $ denotes the
$\Z$-linear hull of a set $S\subset\R^d$. In other words, $\langle
S\rangle_{\Z} $ is the Abelian group generated by $S$. 
\begin{eqnarray*}
\mbox{(Hom$^*$)}&&\mbox{For all finite subsets $F$ of $\langle \varLambda-\varLambda\rangle_{\Z}$, there is a
homothety}\\&&\mbox{$h\!:\, \mathbbm{R}^{d} \rightarrow
\mathbbm{R}^{d}$ such that $h(F)\subset \varLambda$. 
}
\end{eqnarray*}

Translates $\varLambda$ of arbitrary lattices
$L\subset\R^d$ satisfy $\langle \varLambda-\varLambda\rangle_{\Z}=L$
and are thus Delone sets with property 
(Hom$^*$). The following negative results shows that, in order to
obtain positive results on determination for Delone sets $\varLambda$ with property 
(Hom$^*$), one has to impose some restriction on the finite subsets of
$\varLambda$ to
be determined.

\begin{prop}\label{source}
Let $\varLambda\subset\R^d$ be a Delone set with property (Hom$^*$) and let
$U$ be a finite set of pairwise non-parallel
$\varLambda$-directions. Then the finite subsets of $\varLambda$ are not determined by the $X$-rays in the directions of $U$.
\end{prop}
\begin{proof}
Since property (Hom$^*$) is invariant under
translations, we may assume, without loss of generality, that
$0\in\varLambda$. Hence, $\langle\varLambda\rangle_{\Z}\subset \langle \varLambda-\varLambda\rangle_{\Z}$. We argue by induction on $\operatorname{card}U$. The case
$\operatorname{card}U=0$ means $U=\varnothing$ and is obvious. Fix
$k\in\N_{0}$ and suppose the assertion is 
true whenever $\operatorname{card}U=k$. Let $U$ now be a set with 
$\operatorname{card}U=k+1$. By the induction hypothesis, there are
different finite subsets $F$ and $F'$ of $\varLambda$
with the same $X$-rays in the directions of $U'$, where $U'\subset U$
satisfies $\operatorname{card}U'=k$. Let $u$ be the remaining
direction of $U$ and choose a non-zero element $z\in \langle \varLambda-\varLambda\rangle_{\Z}$ parallel to
$u$ such that $z+(F\cup F')$ and $F\cup F'$ are disjoint. Then,
$F'':= F\cup (z+F')$ and
$F''':= F'\cup (z+F)$ are different finite subsets of $\langle \varLambda-\varLambda\rangle_{\Z}$ with the same $X$-rays in the
directions of $U$. By property~(Hom$^*$), there is a homothety $h\!:\,
\R^d \rightarrow \R^d$ such that $h(F''\cup F''')=h(F'')\cup
h(F''')\subset \varLambda$. It follows from Fact~\ref{homotu}(b) that $h(F'')$ and
$h(F''')$ are different finite subsets of $\varLambda$
with the same $X$-rays in the directions of $U$. 
\end{proof}

\begin{rem}\label{source2}
For other versions of the last result, compare~\cite[Theorem
  4.3.1]{GG2} and~\cite[Lemma 2.3.2]{G}. An analysis of the proof of Proposition~\ref{source} shows that, for
any Delone set $\varLambda\subset\R^d$ with property (Hom$^*$) and for 
any finite set $U$ of $k$ pairwise non-parallel
$\varLambda$-directions, there are disjoint finite subsets  $F$ and $F'$ of
$\varLambda$ with
$\operatorname{card}F=\operatorname{card}F'=2^{(k-1)}$ that have the same
$X$-rays in the directions of $U$. Consider any convex subset $C$ of 
$\R^d$ which contains $F$ and $F'$ from above. Then the finite subsets
$F_1:=(C\cap\varLambda)\setminus F$ and
$F_2:=(C\cap\varLambda)\setminus F'$ of
$\varLambda$ also have the same
$X$-rays in the directions of $U$. Whereas the points in $F$ and $F'$
are widely dispersed over a region, those in $F_1$ and $F_2$ are
contiguous in a way similar to atoms in some solid state material.  This procedure
is illustrated in Figure~\ref{fig:contig} in the case of the
 aperiodic cyclotomic model set $\varLambda_{\rm AB}$ as described in
  Example~\ref{exab} below. Note that cyclotomic model sets have
  property (Hom$^*$); see Definition~\ref{algdeldef},
  Remark~\ref{algdellem} and Proposition~\ref{cmsads} below.
\end{rem}

\begin{figure}
\centerline{\epsfysize=0.59\textwidth\epsfbox{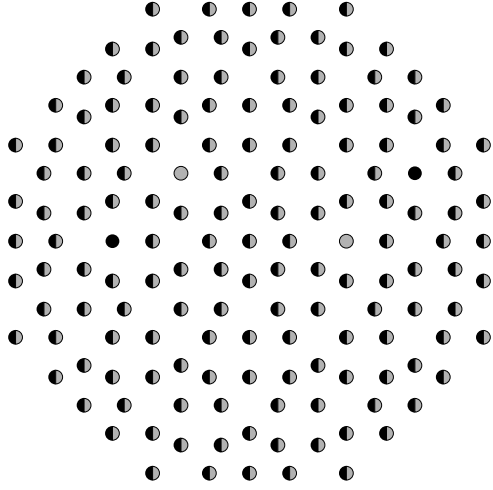}}
\caption{Two contiguous subsets of $\varLambda_{\rm AB}$ with the same
  $X$-rays in the two $\varLambda_{\rm AB}$-directions with slopes $0$ and
  $1$, respectively.}
\label{fig:contig}
\end{figure}

The subsequent positive results on determination are
of limited use in practice since, in general, they do not
comply with the restriction to few high-density directions mentioned
earlier. The first one is well-known and follows from the same
arguments as in the
proof of~\cite[Theorem
4.3.3]{GG2}.

\begin{fact}\label{m+1}
Let $d\geq 2$ and let
$\varLambda$ be a Delone set in $\R^d$. Further, let $U$ be any set of $k+1$ pairwise non-parallel
$\varLambda$-directions where $k\in \mathbbm{N}_{0}$. Then the finite
subsets of 
$\varLambda$ with cardinality less than or equal to $k$ are determined by the $X$-rays in
the directions of $U$. Moreover, all finite subsets $F$ of 
$\varLambda$ with cardinality less than or equal to $k$ satisfy $F=G^{F}_{U}$.
\end{fact}

Let $d\geq
2$ and let $\varLambda$ be a Delone set in $\R^d$ with property (Hom$^*$). Remark~\ref{source2} and Fact~\ref{m+1} show that the finite
subsets of 
$\varLambda$ with cardinality less than or equal to $k$ are determined by
the $X$-rays in
any set of $k+1$ pairwise non-parallel $\varLambda$-directions but not by
 the $X$-rays in $1+\lfloor\log_{2}k\rfloor$ pairwise non-parallel
$\varLambda$-directions. 

\begin{fact}\label{dirdense}
Let $d\geq 2$ and let $\varLambda\subset \R^d$ be relatively dense. Then
the set of $\varLambda$-directions is dense in $\mathbb{S}^{d-1}$.\end{fact}

\begin{prop}\label{bounded}
Let $d\geq 2$, let $r>0$, and let
$\varLambda\subset\R^{d}$ be a Delone set of finite local
complexity. Then, there is a set $U$ of two non-parallel
$\varLambda$-directions such that the subsets of patches of radius $r$ of
$\varLambda$ are determined by the $X$-rays in the directions of
$U$. Moreover, there is a set $U$ of three pairwise non-parallel
$\varLambda$-directions such that, for all subsets $F$ of patches of radius $r$ of
$\varLambda$, one has $F=G^{F}_{U}$.
\end{prop}
\begin{proof}
We denote by $\mathcal{P}_r(\varLambda)$ the collection of subsets of patches of radius $r$ of
$\varLambda$. For the first assertion, note that the finite
local complexity of $\varLambda$ implies that the set $V$ of $\varLambda$-directions $v$ with the
property that there is a set $F\in \mathcal{P}_r(\varLambda)$ and a
line $\ell$ in $\R^d$ in direction $v$ with more than one point of $F$
on $\ell$ is finite. Let $u$ be an arbitrary
$\varLambda$-direction. Then, for every $F\in
\mathcal{P}_{r}(\varLambda)$, one has $F\subset
G^{F}_{\{u\}}\cap\varLambda$. Choose $u''\in\mathbb{S}^{d-1}\cap u^{\perp}$ and note
that, for every $F\in
\mathcal{P}_{r}(\varLambda)$, the orthogonal projection
$(G^{F}_{\{u\}}\cap\varLambda)|u^{\perp}$ of the set
$G^{F}_{\{u\}}\cap\varLambda$ on the hyperplane $u^{\perp}$ is finite
with diameter $D_{u}^{F}<2r$. Moreover, by the finite
local complexity of $\varLambda$, the set $\{D_{u}^{F}\,|\,F\in
\mathcal{P}_{r}(\varLambda)\}$ of diameters is finite. This implies
the existence of a neighbourhood $W$ 
of $u''$ in $\mathbb{S}^{d-1}$ with the
property that, for each line $\ell$ in
 a direction $w\in W$ and any set $F\in
\mathcal{P}_{r}(\varLambda)$, any two elements of the set
$\ell\cap(G^{F}_{\{u\}}\cap\varLambda)$ have a distance less than
$2r$. Since the set of $\varLambda$-directions is dense in
$\mathbb{S}^{d-1}$ by Fact~\ref{dirdense}, and by the finiteness of the
set $V$, this observation shows that one can choose a
$\varLambda$-direction $u'\in W\setminus V$ that is not parallel to $u$. We claim that the elements of $\mathcal{P}_{r}(\varLambda)$ are determined by the $X$-rays in the
directions of $\{u,u'\}$. To this end, let
$F,F'\in\mathcal{P}_{r}(\varLambda)$ satisfy $X_{u}F=X_{u}F'$. Then,
one has
$F,F'\subset G^{F}_{\{u\}}\cap\varLambda$. In order to demonstrate that the identity
$X_{u'}F=X_{u'}F'$ implies the equality $F=F'$, it suffices to show
that each line $\ell$ in direction $u'$ meets at most one element of
$G^{F}_{\{u\}}\cap\varLambda$. Assume the existence of two distinct
elements, say $\lambda$ and
$\lambda'$, in $\ell\cap(G^{F}_{\{u\}}\cap\varLambda)$. Then, by
construction, the distance between $\lambda$ and
$\lambda'$ is less than
$2r$. Hence, $\{\lambda,\lambda'\}\in \mathcal{P}_{r}(\varLambda)$, and further
$u'\in V$, a contradiction. 

For the second part, let $u,u'$ be two
arbitrary non-parallel $\varLambda$-directions and set
$U':=\{u,u'\}$. Note that the finite
local complexity of $\varLambda$ implies that the set $V$ of $\varLambda$-directions $v$ with the
property that there is a set $F\in \mathcal{P}_r(\varLambda)$ and a
line $\ell$ in $\R^d$ in direction $v$ with more than one point of the
finite set $G^{F}_{U'}$
on $\ell$ is finite. Since the set of $\varLambda$-directions is dense in
$\mathbb{S}^{d-1}$ by Fact~\ref{dirdense}, this observation shows that one can choose a
$\varLambda$-direction $u''\notin  V$ that is not parallel to $u$ and
$u'$. By construction, the assertion follows with $U:=\{u,u',u''\}$. 
\end{proof}



\section{Determination of convex subsets of algebraic Delone sets}

\subsection{Algebraic Delone sets}

 For $\varLambda\subset\LC$, we denote by $\mathbbm{K}_{\varLambda}$ the field extension of $\LQ$ that is given by 
$$
\mathbbm{K}_{\varLambda}\,\,:=\,\,\Q\left(\big(\varLambda-\varLambda\big)\cup\big(\overline{\varLambda-\varLambda}\big)\right)\,,
$$
and, further, set 
$
\mathbbm{k}_{\varLambda}:=\mathbbm{K}_{\varLambda}\cap\R$, the maximal
real subfield of $\mathbbm{K}_{\varLambda}$. The following notion will
be useful; see also~\cite{H3,H4} for generalizations and results related to
those presented below.

\begin{defi}\label{algdeldef}
A Delone set $\varLambda\subset\R^2$ is called an {\em algebraic Delone
  set} if it satisfies the following properties:
\begin{eqnarray*}
\mbox{(Alg)}&&\left[\mathbbm{K}_{\varLambda}:\Q\right]<\infty\,.\\
\mbox{(Hom)}&&\mbox{For all finite subsets $F$ of $\mathbbm{K}_{\varLambda}$, there is a
homothety}\\&&\mbox{$h\!:\, \mathbbm{R}^{2} \rightarrow
\mathbbm{R}^{2}$ such that $h(F)\subset \varLambda$\,. 
}
\end{eqnarray*}
\end{defi} 

Translates $\varLambda$ of the square lattice $\Z^2=\Z[i]$ are examples of algebraic Delone sets, with
$\mathbbm{K}_{\varLambda}=\Q(i)$.

\begin{rem}\label{algdellem}
Note that, for any algebraic Delone set $\varLambda$, the field
$\mathbbm{k}_{\varLambda}$ is a real algebraic number field. Trivially,
property~(Hom) for $\varLambda$ implies property~(Hom$^*$).
\end{rem}

Lagarias~\cite{Lagarias} defined the notion of {\em finitely
  generated Delone sets} $\varLambda\subset\R^d$. These are
Delone sets $\varLambda$ with the property that the Abelian group
  $\langle \varLambda-\varLambda\rangle_{\Z}$  is finitely
generated. The last property is always fulfilled by Delone sets of
finite local complexity, which are also called {\em Delone sets of
  finite type}; see~\cite[Theorem 2.1]{Lagarias}. 

\begin{prop}
Let $\varLambda$ be an algebraic Delone set. If
$\varLambda-\varLambda$ is contained in the ring of algebraic
integers, then $\varLambda$ is a finitely generated Delone set.
\end{prop}
\begin{proof}
If
$\varLambda-\varLambda$ is contained in the ring of algebraic
integers, then the Abelian group $\langle \varLambda-\varLambda\rangle_{\Z}$ is a subgroup of the
ring of integers in 
$\mathbbm{K}_{\varLambda}$. Since the latter
 is the maximal order of $\mathbbm{K}_{\varLambda}$
and thus a free Abelian group of finite rank, the assertion follows; cf.~\cite[Ch.~2, Sec.~2]{Bo}.
\end{proof}

\subsection{$U$-polygons in algebraic Delone sets}

The following fact follows immediately from property (Hom) in conjunction with Fact~\ref{homotu}(a). 

\begin{fact}\label{upol4}
Let $\varLambda$ be an algebraic Delone set and let $U\subset
\mathbb{S}^{1}$ be a finite set of directions. Then there is a $U$-polygon in
$\mathbbm{K}_{\varLambda}$ if and only if there is a  $U$-polygon in $\varLambda$.
\end{fact}

\begin{lem}\label{uleq3}
Let $\varLambda$ be an algebraic Delone set. If\/ $U$ is any set of up to $3$ pairwise non-parallel $\varLambda$-directions, then there exists a $U$-polygon in $\varLambda$.
\end{lem}
\begin{proof}
Without loss of generality, we may assume that $\operatorname{card}U=3$. First, construct a triangle in $\mathbbm{K}_{\varLambda}$ having sides parallel to the
given directions of $U$. If two of the vertices are chosen in
$\mathbbm{K}_{\varLambda}$, then the third is automatically in
$\mathbbm{K}_{\varLambda}$. Now, fit six congruent versions of this triangle together in the obvious way
to make an
affinely regular hexagon in $\mathbbm{K}_{\varLambda}$. The latter is
then a $U$-polygon in $\mathbbm{K}_{\varLambda}$ and the assertion  follows from
Fact~\ref{upol4}.
\end{proof}

\begin{prop}\label{characungen}
Let $\varLambda$ be an algebraic Delone set and let $U$ be a set of two or more pairwise non-parallel $\varLambda$-directions. The following statements are equivalent:
\begin{itemize}
\item[(i)]
The convex subsets of $\varLambda$ are determined by the $X$-rays in the directions of $U$.
\item[(ii)]
There is no $U$-polygon in $\varLambda$.
\end{itemize}
\end{prop}
\begin{proof}
For (i) $\Rightarrow$ (ii), suppose the existence of a $U$-polygon $P$
in $\varLambda$. Partition the vertices of $P$ into two
disjoint sets $V,V'$, where the elements of these sets alternate round
the boundary $\operatorname{bd}P$ of $P$. Since $P$ is a $U$-polygon, each line
in the plane parallel to some $u\in U$ that contains a point in $V$
also contains a point in $V'$. In particular, one sees that $\operatorname{card}V=\operatorname{card}V'$. Set
$
C:=(\varLambda\cap P)\setminus (V\cup V')
$. Then, $F:=C\cup V$ and $F':=C\cup V'$ are different convex subsets of $\varLambda$ with the same $X$-rays in the directions of $U$.

For (ii) $\Rightarrow$ (i), suppose the existence of two different
convex subsets of $\varLambda$ with the same $X$-rays in the
directions of $U$. Since the property of being an algebraic Delone set is invariant under translations, we may assume, without loss of generality,
that $0\in\varLambda$, whence $\varLambda\subset \varLambda
-\varLambda$. Then, by the same argumentation as in the proof of
the corresponding direction of~\cite[Theorem 5.5]{GG}, there follows the existence of a $U$-polygon in
$\Q(\varLambda)\subset\mathbbm{K}_{\varLambda}$. More precisely, one has to use Lemma~\ref{uleq3} instead of~\cite[Lemma
4.4]{GG} and note that~\cite[Lemma
5.2]{GG} extends to the more general situation needed here. Fact~\ref{upol4} completes the proof.
\end{proof}

Let $(t_1,t_2,t_3,t_4)$ be an ordered tuple of four pairwise distinct
elements of the set $\mathbbm{R}\cup\{\infty\}$. Then, its {\em cross ratio}
$\langle t_1,t_2,t_3,t_4\rangle$ is the non-zero real number defined by
$$
\langle t_1,t_2,t_3,t_4\rangle := \frac{(t_3 - t_1)(t_4 - t_2)}{(t_3 - t_2)(t_4 - t_1)}\,,
$$
where one uses the usual conventions if one of the $t_i$ equals $\infty$.

\begin{fact}\label{crkn4gen}
For a set $\varLambda\subset\R^2$, the
cross ratio of slopes of four pairwise non-parallel
$\varLambda$-directions is an element of the field $\mathbbm{k}_{\varLambda}$.
\end{fact}

The proof of the following central result uses Darboux's theorem on second 
midpoint polygons; see~\cite{D}, ~\cite{GM} or~\cite[Ch.~1]{G}.

\begin{theorem}\label{finitesetncr0gen}
Let $\varLambda\subset\R^2$, let $U$ be a set of four or more pairwise
non-parallel $\varLambda$-directions, and suppose the existence of a
$U$-polygon. Then the
cross ratio of slopes of any four directions of $U$, arranged in order
of increasing angle with the positive real axis, is an element of the set
\begin{equation}\label{setkl}
\Big(\bigcup_{m\geq 4}\Big ( \bigcup_{\substack{k_3<k_1\leq k_2<k_4\leq m-1\\k_1+k_2=k_3+k_4}} \frac{(1-\zeta_{m}^{k_1})(1-\zeta_{m}^{k_2})}{(1-\zeta_{m}^{k_3})(1-\zeta_{m}^{k_4})}\Big )\Big )\Big )\,\,\cap\,\,
\mathbbm{k}_{\varLambda}\,,\end{equation}
where $\zeta_m:=e^{2\pi i/m}$, a primitive $m$th root of unity in
$\C$.
\end{theorem}
\begin{proof}
The assertion follows from the same arguments as in the first part of the proof
of~\cite[Theorem 4.5]{GG}. Here, one additionally has to employ Fact~\ref{crkn4gen}.\end{proof}

In order to find a sufficient condition for the determination of the convex subsets of an algebraic Delone set $\varLambda$ by
$X$-rays in four pairwise non-parallel $\varLambda$-directions, it is
essential for our approach to gain some insight into the set~\eqref{setkl}. This is done in
the next section. 

\subsection{A cyclotomic theorem}\label{sec9}

We need the following facts from the theory of algebraic number fields. Let $\mathbbm{K}/\mathbbm{k}$ be an extension of
algebraic number fields (i.e., finite extensions of $\LQ$) of degree $d:=[\mathbbm{K}:\mathbbm{k}]$
$\in\N$. Since $\mathbbm{K}/\mathbbm{k}$ is separable, the
corresponding norm $N_{\mathbbm{K}/\mathbbm{k}}\!:\,
\mathbbm{K}\rightarrow \mathbbm{k}$ is given by
$$
N_{\mathbbm{K}/\mathbbm{k}}(\kappa)=\prod_{j=1}^{d}\sigma_{j}(\kappa)\,,
$$
where the $\sigma_{j}$ are the $d$ distinct embeddings of
$\mathbbm{K}/\mathbbm{k}$ into $\C/\mathbbm{k}$; compare~\cite[Algebraic Supplement, Sec.~2,
Corollary 1]{Bo}. In particular, one has $N_{\mathbbm{K}/\mathbbm{k}}(\kappa)=\kappa^d$ for any $\kappa \in \mathbbm{k}$. The norm $N_{\mathbbm{K}/\mathbbm{k}}$ is multiplicative, i.e., for any $\kappa,\lambda\in\mathbbm{K}$, one has
\begin{equation}\label{normmult}
N_{\mathbbm{K}/\mathbbm{k}}(\kappa\lambda)=N_{\mathbbm{K}/\mathbbm{k}}(\kappa) N_{\mathbbm{K}/\mathbbm{k}}(\lambda)\,.
\end{equation}
In particular, the norm $N_{\mathbbm{K}/\mathbbm{k}}$ induces a homomorphism from the multiplicative subgroup $\mathbbm{K}^{\times}$ of $\mathbbm{K}$ to the multiplicative subgroup $\mathbbm{k}^{\times}$ of $\mathbbm{k}$. Moreover, the norm is transitive in the following sense. If $\mathbbm{L}$ is any intermediate field of $\mathbbm{K}/\mathbbm{k}$
above, then one has
\begin{equation}\label{normtr}
N_{\mathbbm{K}/\mathbbm{k}}=N_{\mathbbm{L}/\mathbbm{k}}\circ N_{\mathbbm{K}/\mathbbm{L}}\,.
\end{equation}

We also need the following facts from the theory of
$p$-adic valuations; compare~\cite{Gou,ko} for details. Let $p\in \mathbbm{P}$. The $p$-adic valuation on $\Z$ is the
function $v_{p}$, defined by $v_{p}(0):=\infty$ together with the equation
$$
n=p^{v_{p}(n)}n'
$$
for $n\neq 0$, where $p$ does not divide $n'$; that is, $v_{p}(n)$ is
the exponent of the highest power of $p$ that divides $n$. The function
$v_{p}$ is extended to $\Q$ by defining
$$
v_{p}\Big(\frac{a}{b}\Big):= v_{p}(a)-v_{p}(b)
$$
for $a,b\in\Z\setminus\{0\}$. Note that $v_{p}$
is $\Z$-valued on $\Q\setminus\{0\}$. Further, $v_{p}$ can be extended to the algebraic closure $\Q_{p}^{\operatorname{alg}}$
of the field $\Q_{p}$ of $p$-{\em adic numbers},
containing $\Q$. Note that $\Q_{p}^{\operatorname{alg}}$ contains the algebraic
closure $\Q^{\operatorname{alg}}$ of $\Q$ and hence all algebraic numbers. On
$\Q_{p}^{\operatorname{alg}}\setminus\{0\}$, $v_{p}$ takes values in $\Q$, and
satisfies
\begin{eqnarray}
v_{p}(-x)&=&v_{p}(x)\,,\label{minuspreserve}\\
v_{p}(xy)&=&v_{p}(x)+v_{p}(y)\,,\label{log1}\\
v_{p}\Big(\frac{x}{y}\Big)&=&v_{p}(x)-v_{p}(y)\label{log2}
\end{eqnarray} 
and
\begin{eqnarray}
v_{p}(x+y)&\geq& \operatorname{min}\{v_{p}(x),v_{p}(y)\}\,.
\end{eqnarray}

For $n\in
\mathbbm{N}$, we always let $\zetan := e^{2\pi i/n}$, as a specific
choice for a
primitive $n$th root of unity in $\C$. Further, $\phi$ will always denote Euler's totient function, i.e., $$\phi(n) =
\operatorname{card}\left(\big\{k \in \mathbbm{N}\, |\,1 \leq k \leq n
  \textnormal{ and } \operatorname{gcd}(k,n)=1\big\}\right)\,.$$

\begin{fact}\cite[Proposition 3.6]{GG}\label{ppuou} 
Let $p\in \mathbbm{P}$ and let $r,s,t \in \mathbbm{N}$.  
If $r$ is not
a power of $p$ and $\operatorname{gcd}(r,s)=1$, one has 
\begin{equation}
v_{p}(1-\zeta_{r}^{s})=0\,.
\end{equation} 
Otherwise, if $\operatorname{gcd}(p,s)=1$, then
\begin{equation}
v_{p}(1-\zeta_{p^{t}}^{s})=\frac{1}{p^{t-1}(p-1)}\,.
\end{equation}
\end{fact}

\begin{defi}
Let $k,m\in\mathbbm{N}$ and let $p\in\mathbbm{P}$. An $m$th root of unity $\zeta_{m}^{k}$ is called a $p${\em -power root
  of unity} if there is a $t\in\mathbbm{N}$ such that
$\frac{k}{m}=\frac{s}{p^t}$ for some $s\in\mathbbm{N}$
  with $\operatorname{gcd}(p,s)=1$. 
\end{defi}

Note that an $m$th root of unity $\zeta_{m}^{k}$ is a $p$-power root
  of unity if and only if it is a primitive $p^{t}$th root of unity
  for some $t\in\mathbbm{N}$. The following property is immediate.

\begin{fact}\label{ppuou2}
Let $k,t \in\mathbbm{N}$ and 
$p\in\mathbbm{P}$. Further, let $j,m \in\mathbbm{N}$ with
$\operatorname{gcd}(j,m)=1$. Then, 
$\zeta_{m}^{k}$ is a primitive $p^t$th root of unity
 if and only if $(\zeta_{m}^{j})^{k}$ is a
 primitive $p^t$th root of unity.
 \end{fact}

\begin{fact}[Gau\ss]\cite[Theorem 2.5]{Wa}\label{gau}
The $n$th cyclotomic field $\Q(\zetan)$ is of degree $[\Q(\zeta_n) :
  \mathbbm{Q}] = \phi(n)$ over $\LQ$. The field extension $\Q(\zeta_n)/ \mathbbm{Q}$
is a Galois extension with Abelian Galois group $G(\Q(\zeta_n)/
\mathbbm{Q}) \simeq (\Z / n\Z)^{\times}$,
where $a\, (\textnormal{mod}\, n)$ with $\operatorname{gcd}(a,n)=1$ corresponds to the automorphism given by\/ $\zetan \mapsto \zetan^{a}$.
\end{fact}

\begin{lem}\label{sigmapreserve0}
Let $m,k\in\mathbbm{N}$ and let $p\in\mathbbm{P}$. If $\sigma\in G(\LQ(\zeta_m)/\Q)$, then
$$
v_{p}(1-\zeta_{m}^{k})=v_{p}\big(\sigma(1-\zeta_{m}^{k})\big)\,.
$$
\end{lem}
\begin{proof}
By Fact~\ref{gau}, $\sigma$ is given by $\zeta_{m}\mapsto
\zeta_{m}^{j}$, where $j\in\mathbbm{N}$ satisfies $\operatorname{gcd}(j,m)=1$. The
assertion follows immediately from Fact~\ref{ppuou} in
conjunction with Fact~\ref{ppuou2}.
\end{proof}

\begin{defi}\label{fmddefi}
Let $m\geq 4$ be a natural number. We define 
$$
D'_{m}:=\left\{\left.(k_1,k_2,k_3,k_4)\in \mathbbm{N}^4 \right|
k_1,k_2,k_3,k_4\leq m-1 \mbox{ and } k_1+k_2=k_3+k_4\right\}\,,
$$
together with its subset
$$
D_{m}:=\left\{\left.(k_1,k_2,k_3,k_4)\in \mathbbm{N}^4 \right| k_3<k_1\leq
k_2<k_4\leq m-1 \mbox{ and } k_1+k_2=k_3+k_4\right\}\,,
$$
and define the function $f_{m}\,:\, D'_{m}\rightarrow \Q(\zeta_m)^{\times}$ by
\begin{equation}\label{fmd}
f_{m}\big((k_1,k_2,k_3,k_4)\big):=\frac{(1-\zeta_{m}^{k_1})(1-\zeta_{m}^{k_2})}{(1-\zeta_{m}^{k_3})(1-\zeta_{m}^{k_4})}\,.
\end{equation}
\end{defi}

In fact, the function $f_m$ is real-valued; see the proof of~\cite[Lemma
3.1]{GG}. Together with Fact~\ref{gau}, this 
shows that, for any $d\in D'_{m}$, the field $\Q(f_m(d))$ is a real
algebraic number field. Without further mention, we shall use this
fact in the following. 

\begin{lem}\label{sigmapreserve}
Let $m\geq 4$, let $p\in \mathbbm{P}$, and let $d\in
D'_{m}$. If $\sigma\in G(\LQ(\zeta_m)/\Q)$, then
$$
v_{p}\big(f_{m}(d)\big)=v_{p}\big(\sigma\big(f_{m}(d)\big)\big)\,.
$$
\end{lem}
\begin{proof}
The
assertion follows from Lemma~\ref{sigmapreserve0} together with~(\ref{log1})
and~(\ref{log2}).
\end{proof}

\begin{fact}\cite[Ch.~5.2, Theorem 2.8]{La}\label{extension}
Let $\sigma\!:\, \mathbbm{K}\rightarrow \mathbbm{K}'$ be an isomorphism of fields, let
$\mathbbm{E}$
be an algebraic extension of $\mathbbm{K}$, and let $\mathbbm{L}$ be an algebraically closed
extension of $\mathbbm{K}'$. Then, there exists a field homomorphism  $\sigma'\;:\;
\mathbbm{E}\rightarrow \mathbbm{L}$ which extends $\sigma$.
\end{fact}

\begin{lem}\label{1dn}
Let $m\geq 4$ and let
$d\in D'_{m}$. Then, for any prime
factor $p\in\mathbbm{P}$ of the numerator of the field norm 
$N_{\Q(f_m(d))/\Q}(f_m(d))$, one has 
\begin{equation*}\label{vpn}
v_{p}\big(N_{\Q(f_{m}(d))/\Q}\big(f_{m}(d)\big)\big)=e\,v_{p}\big(f_{m}(d)\big)\in \mathbbm{N}\,,
\end{equation*}
where $e:=[\Q(f_m(d)):\Q]\in\N$ is the degree of $f_{m}(d)$ over $\Q$. 
\end{lem}
\begin{proof}
Consider the inclusion of fields $\Q\big(f_{m}(d)\big)\subset
\Q(\zeta_m)$. The norm 
$N_{\Q(f_{m}(d))/\Q}\!:\,\Q\big(f_{m}(d)\big)\rightarrow \Q$ of the Galois
extension $\Q(f_{m}(d))/\Q$ is given by
$$N_{\Q(f_{m}(d))/\Q}(q)=\prod_{j=1}^{e}\sigma_{j}(q)$$ for $q\in
\Q(f_{m}(d))$, where
$\{\sigma_{1},\dots,\sigma_{e}\}$ is the underlying set of the Galois group
$G(\Q(f_{m}(d))/\Q)$. Note that
the field extension $\Q(f_{m}(d))/\Q$ is indeed a Galois
extension since, by Fact~\ref{gau}, the Galois extension
  $\Q(\zeta_m)/\Q$ has an Abelian Galois group. Moreover,
 each field automorphism $\sigma_{j}\in G(\Q(f_{m}(d))/\Q)$ can be extended to
a field automorphism $\sigma'_{j}\in G(\Q(\zeta_m)/\Q)$ by 
Fact~\ref{extension}. It
follows that
$$N_{\Q(f_{m}(d))/\Q}\big(f_{m}(d)\big)=\prod_{j=1}^{e}\sigma'_{j}\big(f_{m}(d)\big)\,.$$ Using the $p$-adic valuation $v_{p}$ in conjunction with~(\ref{log1}) and Lemma~\ref{sigmapreserve}, the assertion follows. 
\end{proof}


The multiplicativity of norms immediately implies the following

\begin{fact}\label{3dn}
Let $m\geq 4$ and let
$d:=(k_1,k_2,k_3,k_4)\in D'_{m}$. Then, one has
$d':=(k_3,k_4,k_1,k_2)\in D'_{m}$ and $f_{m}(d')=1/f_{m}(d)$, whence $\Q(f_{m}(d))=\Q(f_{m}(d'))$. Further, one has the identity 
$$
N_{\Q(f_{m}(d))/\Q}\big(f_{m}(d')\big)=\big(N_{\Q(f_{m}(d))/\Q}\big(f_{m}(d)\big)\big)^{-1}\,.
$$
\end{fact}

The following main result of this section may be viewed as a generalization of the first
part of~\cite[Theorem 3.10]{GG}.  

\begin{theorem}\label{intersectk8general}
For any real algebraic number field $\kk$, the set  
$$
N_{\mathbbm{k}/\Q}\Big ( \big(\bigcup_{m\geq 4} f_{m}(D_{m})\big )\cap 
 \mathbbm{k}\Big ) 
$$
is a finite subset of $\Q$.
\end{theorem}
\begin{proof}
Set $f:=[\kk:\Q]\in\N$ and let $f_{m}(d)\in (\bigcup_{m\geq 4} f_{m}(D'_{m}))\cap 
 \mathbbm{k}$ for suitable $m\geq 4$ and  $d=(k_1,k_2,k_3,k_4)\in
D'_{m}$. Since $\Q(f_{m}(d))$ is an intermediate field of the field extension $\mathbbm{k}/\LQ$, $f_{m}(d)$ is of degree $e:=[\Q(f_{m}(d)):\Q]$ over $\Q$, where $e$
is a divisor of $f$. Since $f_{m}(d)\neq
0$,  
 the norm $N_{\Q(f_{m}(d))/\Q}(f_{m}(d))$ is non-zero, hence its absolute value is greater
 than zero. Suppose that the absolute value of
 $N_{\Q(f_{m}(d))/\Q}(f_{m}(d))$ is greater than $1$. Then the numerator of
$N_{\Q(f_{m}(d))/\Q}(f_{m}(d))$ has a prime factor, say
$p\in\mathbbm{P}$. Further, by Lemma~\ref{1dn}, for every such prime factor
$p\in\mathbbm{P}$, one has
\begin{equation}\label{restr17}
v_{p}\big(N_{\Q(f_{m}(d))/\Q}\big(f_{m}(d)\big)\big)= e\,v_{p}\big(f_{m}(d)\big)=e\,v_{p}\left ( \frac{(1-\zeta_{m}^{k_1})(1-\zeta_{m}^{k_2})}{(1-\zeta_{m}^{k_3})(1-\zeta_{m}^{k_4})}
\right )\in \mathbbm{N}\,.
\end{equation}
Applying~\eqref{log1}, \eqref{log2} and Fact~\ref{ppuou},
one sees that $v_{p}(f_{m}(d))$ is a sum of at most four terms of the
form $1/(p^{t'-1}(p-1))$ for various $t'\in\N$, with one or two
positive terms and at most two negative ones. Let $t$ be the smallest
$t'$ occurring in one of the positive terms. Then, \eqref{restr17} shows in particular that   
$
2e/(p^{t-1}(p-1))\geq 1
$
or, equivalently,
\begin{equation}\label{restr171}
p^{t-1}(p-1)\leq 2\,e\,.
\end{equation}
One can now see that, by~\eqref{restr17} and~\eqref{restr171} and the obvious fact that
$
p^{t-1}(p-1)\rightarrow \infty
$
for fixed $p\in\mathbbm{P}$ as $t\rightarrow\infty$ (resp., for
fixed $t\in\N$ as $\mathbbm{P}\owns p\rightarrow\infty$), there are
only a finite number of cases to deal with. In particular, one sees
that the numerator of $N_{\Q(f_{m}(d))/\Q}(f_{m}(d))$ can only attain
finitely many values. Since, by assumption, the absolute value of
$N_{\Q(f_{m}(d))/\Q}(f_{m}(d))$ is greater than $1$, the existence of
a finite set of non-zero rational numbers follows, say $N_{e}$,
such that $N_{\Q(f_{m}(d))/\Q}(f_{m}(d))\in N_{e}$. Moreover, it follows from Fact~\ref{3dn} that, if the absolute value of 
$N_{\Q(f_{m}(d))/\Q}(f_{m}(d))$ is smaller than one, then one has
$N_{\Q(f_{m}(d))/\Q}(f_{m}(d))\in(N_{e})^{-1}$, whereas the
missing case can only lead to the norms $\pm 1$. The transitivity of norms (cf.~\eqref{normtr}) in conjunction with the multiplicativity of norms (cf.~\eqref{normmult}) immediately gives $N_{\mathbbm{k}/\Q}(f_{m}(d))\in
(\{1,-1\}\cup N_{e}\cup(N_{e})^{-1})^{f/e}$. Since $D_{m}\subset D'_{m}$ and since the above analysis only depends on the degree $e$ of $f_m(d)$ over $\LQ$, one obtains   
$$
N_{\mathbbm{k}/\Q}\Big ( \big(\bigcup_{m\geq 4} f_{m}(D_{m})\big )\cap 
 \mathbbm{k}\Big ) 
\,\,\subset\,\,\bigcup_{e|f}\big(\{1,-1\}\cup
N_{e}\cup(N_{e})^{-1}\big)^{\frac{f}{e}}\,,
$$
for suitable finite subsets $N_e$ of $\LQ\setminus\{0\}$. The assertion follows. 
\end{proof}

\subsection{Results}

\begin{theorem}\label{finitesetncrsupergeneral}
For all algebraic Delone sets $\varLambda$, there is a finite set $N_{\varLambda}\subset \Q$ such that,
for all sets $U$ of four or more pairwise
non-parallel $\varLambda$-directions, one has the following: If there
exists a
$U$-polygon, then the
cross ratio of slopes of any four directions of $U$, arranged in order
of increasing angle with the positive real axis, maps under the norm
$N_{\mathbbm{k}_{\varLambda}/\Q}$ to $N_{\varLambda}$.
\end{theorem}
\begin{proof}
The assertion is an immediate consequence of Remark~\ref{algdellem} and Theorems~\ref{finitesetncr0gen} and~\ref{intersectk8general}.
\end{proof}

Summing up, we can now state our main result. 

\begin{theorem}\label{main3gen}
Any algebraic Delone set $\varLambda$ has the following properties:
\begin{itemize}
\item[(a)]
There is a finite set $N_{\varLambda}\subset \Q$ such that,
 for all sets $U$ of four pairwise
non-parallel $\varLambda$-directions, one has the following: If $\,U$ has the property that the cross ratio of slopes of the
directions of $U$, arranged in order of increasing angle with the
positive real axis, does not map under the norm
$N_{\mathbbm{k}_{\varLambda}/\Q}$ to $N_{\varLambda}$, then the convex subsets of $\varLambda$ are determined by the $X$-rays in the directions of $U$.
\item[(b)]
For all sets $U$  of three or less pairwise
non-parallel $\varLambda$-directions, the convex subsets of $\varLambda$ are not determined by the $X$-rays in the
directions of $U$. 
\end{itemize}
\end{theorem}
\begin{proof}
Part (a) follows immediately from Proposition~\ref{characungen} in
conjunction with 
Theorem~\ref{finitesetncrsupergeneral}. Assertion (b) is an immediate consequence of Proposition~\ref{characungen} in conjunction with Lemma~\ref{uleq3}. 
\end{proof}

\subsection{Application to cyclotomic model sets and an open problem}\label{applopen}

Following Moody~\cite{Moody}, modified along the lines of the algebraic setting of Pleasants~\cite{PABP}, we
define as follows.

\begin{defi}\label{deficy}
Let $n\geq 3$ and let $.^{\star}\!:\,
  \Z[\zeta_n]\rightarrow (\R^2)^{\phi(n)/2-1}$ be any map of the form $z\mapsto
  (\sigma_{2}(z),\dots,\sigma_{\phi(n)/2}(z))$, where the
set $\{\sigma_{2},\dots,\sigma_{\phi(n)/2}\}$ arises from
$G(\Q(\zeta_n)/ \mathbbm{Q})\setminus\{\operatorname{id},\bar{.}\}$ by choosing exactly one automorphism
from each pair of complex conjugate ones
(cf.~Fact~\ref{gau}). Then, for any such choice, each translate
  $\varLambda$ of
$$
\varLambda(W):=\{z\in\Z[\zeta_n]\,|\,z^{\star}\in W\}\,,
$$  
where $W\subset
  (\R^2)^{\phi(n)/2-1}$ is any set with $\varnothing\,
  \neq\, \operatorname{int}W\subset W\subset \operatorname{cl}(\operatorname{int}W)$ and
  $\operatorname{cl}(\operatorname{int}W)$ compact, is a \emph{cyclotomic model set with underlying
  $\Z$-module $\Z[\zeta_n]$}. Moreover, $.^{\star}$ and $W$ are called the \emph{star
    map} and \emph{window} of $\varLambda$, respectively. 
\end{defi}

\begin{rem}\label{propms1} 
In
  Definition~\ref{deficy}, we use the convention that, for $\phi(n)=2$ (i.e., $n\in\{3,4,6\}$), $(\R^2)^{\phi(n)/2-1}$ is the trivial Abelian group $\{0\}$ and the star map is the zero
  map. Note that the ring $\Z[\zeta_n]$ is the ring of integers in
  the $n$th cyclotomic field and is a dense subset of the plane if $n\notin\{3,4,6\}$; cf.~\cite[Theorem
    2.6]{Wa}. Model sets $\varLambda\subset\R^d$ are Meyer sets and
  have many other fascinating 
long-range order properties; see~\cite{Moody,Schl2,Schl} for the
  general setting and details. Further, a cyclotomic model set $\varLambda$ with underlying
$\Z$-module $\Zn$ is aperiodic if and only if
$n\notin\{3,4,6\}$, i.e., the translates of the square
(resp., triangular) lattice are the only examples of cyclotomic model
 sets with non-zero translation symmetries; cf.~\cite{BG2,H0,H} for
 more details and properties of (cyclotomic) model sets.
\end{rem}

\begin{ex}\label{exab}
 For an (aperiodic) cyclotomic model set with
 underlying $\Z$-module $\Z[\zeta_8]$, consider
 $\varLambda_{\text{AB}} := \{z \in \Z[\zeta_8]\, | \,z^{\star} \in
 W\}$, 
where the star map $.^{\star}$ is the Galois automorphism
 in $G(\Q(\zeta_8)/ \mathbbm{Q})$, 
defined by $\zeta_{8} \mapsto \zeta_{8}^3$ (cf.~Fact~\ref{gau}), and the window $W$ is
the regular octagon centred at the origin, with vertices in the
directions that arise from the $8$th roots of unity by a rotation
through $\pi/8$, and of unit edge length; see~\cite{am,bj,ga}. Then
$\varLambda_{\rm AB}$ is associated with the well-known Ammann-Beenker
tiling of the plane with squares and rhombi, both having edge length
$1$; see Figure~\ref{fig:ab} for an illustration. Other examples of
aperiodic 
cyclotomic model sets with
 underlying $\Z$-module $\Z[\zeta_n]$ are the vertex sets of the
 T\"ubingen triangle tiling~\cite{bk1,bk2} ($n=5$) and the shield
 tiling~\cite{ga} ($n=12$); see~\cite{H0,H} for details. 
\end{ex}

\begin{figure}
\centerline{\epsfysize=0.5\textwidth\epsfbox{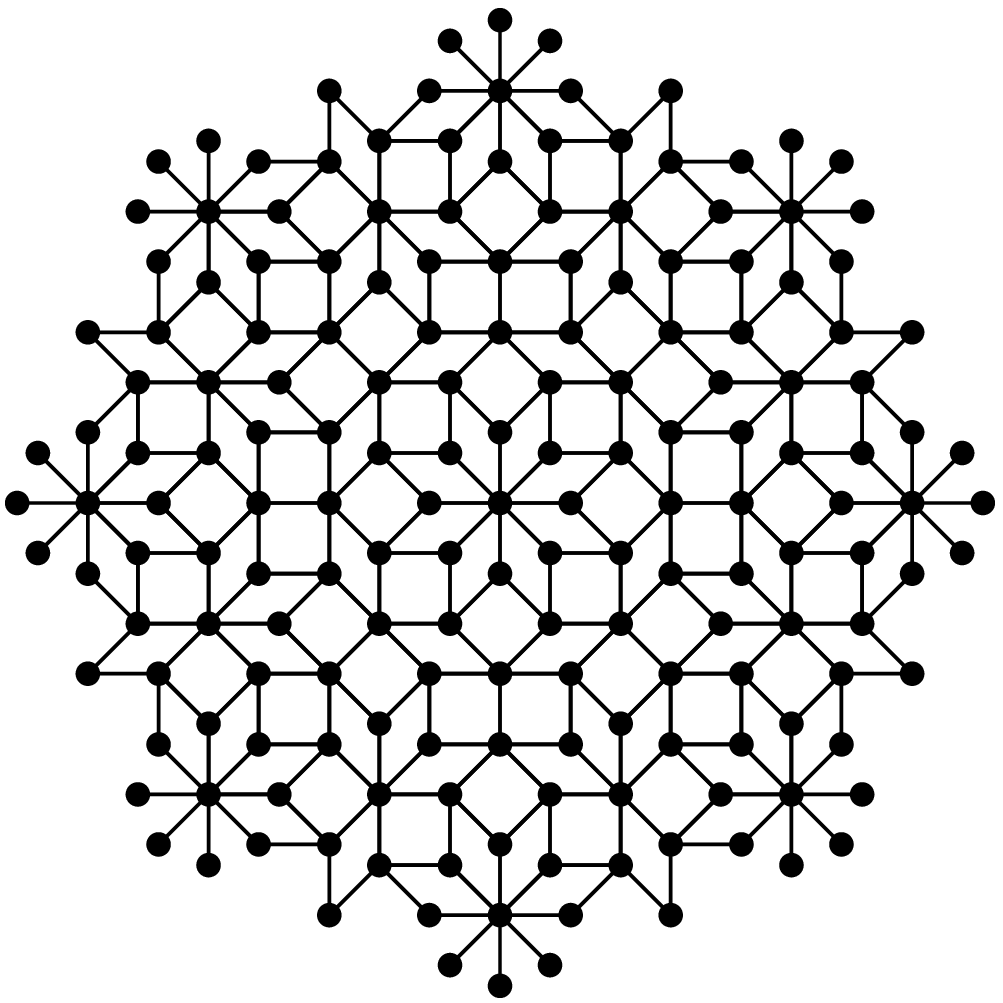}\hspace{0.2\textwidth}
\epsfysize=0.5\textwidth\epsfbox{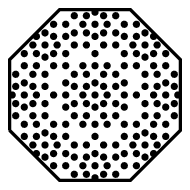}}
\caption{A central patch of the eightfold Ammann-Beenker tiling with vertex set $\varLambda_{\rm AB}$ (left) and the $.^{\star}$-image of $\varLambda_{\rm AB}$ inside the octagonal window (right), with relative scale as described in the text.}
\label{fig:ab}
\end{figure}

The following result on the maximal real subfield
$\Q(\zeta_n+\bar{\zeta}_n)$ of the $n$th cyclotomic field follows
immediately from Fact~\ref{gau}. 

\begin{fact}
If $n\geq 3$, one has $[\Q(\zeta_n+\bar{\zeta}_n) : \mathbbm{Q}] =
\phi(n)/2$.
\end{fact}

Note that the ring $\Z[\zeta_n+\bar{\zeta}_n]$ is the ring of integers in
$\Q(\zeta_n+\bar{\zeta}_n)$; cf.~\cite[Proposition 2.16]{Wa}. A real algebraic integer $\lambda$ is called a {\em Pisot-Vijayaraghavan} number ({\em
  PV-number}) if $\lambda>1$ while all (algebraic) conjugates of
$\lambda$ have moduli strictly less than $1$. The following fact follows from~\cite[Ch.~1, Theorem 2]{Sa}.

\begin{fact}\label{pisot}
For $n\geq 3$, there is a PV-number of degree $\phi(n)/2$ in $\Z[\zeta_n+\bar{\zeta}_n]$.
\end{fact}

For $n\geq 3$ and $\lambda\in\Z[\zeta_n+\bar{\zeta}_n]$, we denote by $m_{\lambda}^{\star}$ the $\Z$-module
endomorphism of $\Z[\zeta_n]^{\star}$ which is given by
$m_{\lambda}^{\star}(z^{\star})=(\lambda z)^{\star}$, where $z\in
\Z[\zeta_n]$ and $.^{\star}$ is a star map of a cyclotomic model set with underlying $\Z$-module $\Z[\zeta_n]$.

\begin{lem}\label{r2}
Let $n\,\in\,\mathbbm{N}\setminus \{1,2,3,4,6\}$, and let $.^{\star}$
be a star map of a cyclotomic model set with underlying $\Z$-module $\Z[\zeta_n]$. Then, for any
PV-number $\lambda$ of degree $\phi(n)/2$ in
$\Z[\zeta_n+\bar{\zeta}_n]$, a suitable power of $m_{\lambda}^{\star}$
is contractive, i.e., there is a $k\in\N$ and a 
$\xi \in (0,1)$ such that $\Arrowvert (m_{\lambda}^{\star})^k(z^{\star})\Arrowvert\leq \xi\, \Arrowvert z^{\star}\Arrowvert$ holds for all $z\in \Z[\zeta_n]$.
\end{lem}
\begin{proof} 
Since all norms on $\R^d$ are equivalent, it suffices to prove the
assertion in case of the maximum norm $\Arrowvert \cdot
\Arrowvert_{\infty}$ on $(\mathbbm{R}^{2})^{\phi(n)/2-1}$ with respect to the
Euclidean norm on $\R^2$ rather than considering the Euclidean norm $\Arrowvert \cdot
\Arrowvert$ on
$(\mathbbm{R}^{2})^{\phi(n)/2-1}$ itself. But in that case, the assertion
follows immediately with $k:=1$ and $$\xi:=\operatorname{max}\{\lvert \sigma_{j}(\lambda)\rvert\,|
\, j\in\{2,\dots,\phi(n)/2\}\}\,,$$ since the set
$\{\lambda,\sigma_{2}(\lambda),\dots,\sigma_{\phi(n)/2}(\lambda)\}$
equals the set of 
conjugates of $\lambda$. 
\end{proof}

\begin{lem}\label{rr2}
Let $n\geq 3$ and let $.^{\star}$ be a star map of a cyclotomic model
set with underlying $\Z$-module $\Z[\zeta_n]$. Then
$\Z[\zeta_n]^{\star}$ is dense in $(\R^2)^{\phi(n)/2-1}$.  
\end{lem}
\begin{proof}
If $n\in\{3,4,6\}$, one even has
$\Z[\zeta_n]^{\star}=(\R^2)^{\phi(n)/2-1}=\{0\}$. Otherwise, let $\lambda$ be a PV-number 
of degree $\phi(n)/2$ in $\Z[\zeta_n+\bar{\zeta}_n]$; cf.~Fact~\ref{pisot}. Then, since
$\Z[\zeta_n]$ is the maximal order of $\Q(\zeta_n)$ by
Remark~\ref{propms1}, for any $k\in\N$, the set $$\{(\lambda^k
z,(m_{\lambda}^{\star})^k(z^{\star}))\,|\,z\in\Z[\zeta_n]\}$$ is a (full) lattice in
$\R^{2}\times(\R^2)^{\phi(n)/2-1}$; cf.~\cite[Ch.~2,
Sec.~3]{Bo}. In conjunction with Lemma~\ref{r2}, this implies that,
for any $\varepsilon>0$, the $\Z$-module $\Z[\zeta_n]^{\star}$
contains an $\R$-basis of $(\R^2)^{\phi(n)/2-1}$ whose elements have
norms $\leq\varepsilon$. The assertion follows.
\end{proof}

\begin{lem}\label{dilate}
Let $n\geq 3$ and let
$\varLambda$ be a cyclotomic model
set with underlying $\Z$-module $\Z[\zeta_n]$. Then, for any finite set $F\subset \Q(\zeta_n)$, there is a homothety $h\!:\, \R^2 \rightarrow \R^2$ such that $h(F)\subset \varLambda$.
\end{lem}
\begin{proof}
Without loss of generality, we may assume that $\varLambda$ is of the
form $\varLambda(W)$ (see Definition~\ref{deficy}) and, further, that $F\neq\varnothing$. By~\cite[Ch.~7.1, Proposition 1.1]{La} in conjunction
with Fact~\ref{gau} and Remark~\ref{propms1}, there is an $l\in
\mathbbm{N}$ with $l F \subset \Z[\zeta_n]$. Let $.^{\star}$ be the star map that is used in
 the construction of $\varLambda(W)$. If $n\in\{3,4,6\}$, we are done by
 defining the homothety $h\!:\, \R^2 \rightarrow \R^2$ by
$
z\mapsto lz
$. Otherwise, by 
$\operatorname{int}W\neq \varnothing$ in conjunction with Lemma~\ref{rr2}, there follows the existence of a
suitable $z_{0}\in \Z[\zeta_n]$ with $z_{0}^{\star}\in
\operatorname{int}W$. Consider the open neighbourhood $V:= \operatorname{int}W -  z_{0}^{\star}$ of $0$ in
$(\R^2)^{\phi(n)/2-1}$. Next, choose a PV-number $\lambda$
of degree $\phi(n)/2$ in $\Z[\zeta_n+\bar{\zeta}_n]$;
cf.~Fact~\ref{pisot}. By virtue of Lemma~\ref{r2}, there is a $k\in\mathbbm{N}$ such that $$(m_{\lambda}^{\star})^{k}((lF)^{\star})\subset V\,.$$ It follows that $\{(\lambda^{k} z + z_{0})^{\star}\, |\, z\in lF\}\subset \operatorname{int}W$ and, further, that $h(F)\subset \varLambda(W)$, where $h\!:\, \R^2 \rightarrow \R^2$ is the homothety given by $z \mapsto (l\lambda^{k}) z + z_{0}$.
\end{proof}   

\begin{fact}\label{klkl}
Let $n\geq 3$ and let
$\varLambda$ be a cyclotomic model
set  with underlying $\Z$-module $\Z[\zeta_n]$. Then, one has
$\mathbbm{K}_{\varLambda}\subset\Q(\zeta_n)$ and thus $\mathbbm{k}_{\varLambda}\subset\Q(\zeta_n+\bar{\zeta}_n)$. 
\end{fact}

\begin{prop}\label{cmsads}
Cyclotomic model sets are algebraic Delone
sets.
\end{prop}
\begin{proof}
Any cyclotomic model
set is a Meyer
set by Remark~\ref{propms1}. Properties (Alg) and (Hom) follow
immediately from Fact~\ref{klkl} in conjunction with Fact~\ref{gau}
and Lemma~\ref{dilate}, respectively. 
\end{proof}

As another immediate consequence of Lemma~\ref{dilate}, one verifies
the following

\begin{fact}
Let $n\geq 3$ and let $\varLambda$ be a cyclotomic model
set with underlying $\Z$-module $\Z[\zeta_n]$. Then the set of 
$\varLambda$-directions is precisely the set of
 $\Zn$-directions. 
\end{fact}

An analysis of the proof of Theorem~\ref{main3gen} gives the following

\begin{theorem}\label{main3gencoro}
For all $n\geq 3$, there is a finite set $N_{n}\subset \Q$
 such that, for all cyclotomic model sets $\varLambda$ with underlying
 $\Z$-module $\Z[\zeta_n]$ and all sets $U$ of four pairwise
non-parallel $\Zn$-directions, one has the following: If $U$ has the property that the cross ratio of slopes of the
directions of $U$, arranged in order of increasing angle with the
positive real axis, does not map under the norm
$N_{\Q(\zeta_n+\bar{\zeta}_n)/\Q}$ to $N_{n}$, then the convex subsets of $\varLambda$ are determined by the $X$-rays in the directions of $U$.
\end{theorem}

The above analysis allows the construction of specific sets $U$ of four pairwise
non-parallel $\Zn$-directions having the property that, for all
cyclotomic model sets $\varLambda$ with underlying $\Z$-module
$\Z[\zeta_n]$, the convex subsets of $\varLambda$ are determined by
the corresponding $X$-rays. For example, the sets of $\Zn$-directions 
parallel to the elements of the following sets have this property:
$U_n:=\{1,1+\zeta_n,1+2\zeta_n,1+5\zeta_n\}$,
$U_n':=\{1,2+\zeta_n,\zeta_n,-1+2\zeta_n\}$ and
$U_n'':=\{2+\zeta_n,3+2\zeta_n,1+\zeta_n,2+3\zeta_n\}$; cf.~\cite[Theorem
  2.54]{H} or~\cite[Theorem 15]{H0} and compare~\cite[Theorem 5.7 and
  Remark 5.8]{GG}. Often, one can even find examples that yield dense
lines in the corresponding discrete structures. For example, for the
practically relevant quasicrystallographic case of (aperiodic)
cyclotomic model sets $\varLambda$ with underlying $\Z$-module
$\Z[\zeta_n]$, where $n=5,8,10,12$, this is achieved by the sets of $\Zn$-directions 
parallel to the elements of the following sets, where $\tau$ denotes
the golden ratio (i.e., $\tau=(1+\sqrt{5})/2$):
$U_8:=\{1+\zeta_{8},(-1+\sqrt{2})+\sqrt{2}\zeta_{8},(-1-\sqrt{2})+\zeta_{8},-2+(-1+\sqrt{2})\zeta_8\}$, $
U_{5}:=U_{10}:=\{(1+\tau)+\zeta_{5},(\tau-1)+\zeta_{5},-\tau+\zeta_{5},2\tau-\zeta_{5}\}
$ and $
U_{12}:=\{1,2+\zeta_{12},\zeta_{12},\sqrt{3}-\zeta_{12}\}$, respectively; cf.~\cite[Theorem 2.56, Example 2.57 and Remark 2.58]{H} or~\cite[Theorem 16, Example 3 and Remark 40]{H0}. Note that orders $5,8,10$ and $12$ occur as standard
cyclic symmetries of genuine quasicrystals; cf.~\cite{St}.

\begin{prob}
In the crystallographic case of the square lattice $\Z[i]$, Gardner and Gritzmann were able to
show that the convex subsets of $\Z[i]$ are determined by the $X$-rays
in the directions of \emph{any} set $U$ of seven pairwise
non-parallel $\Z[i]$-directions; cf.~\cite[Theorem 5.7]{GG}. It would
be interesting to know if, for all cyclotomic model sets $\varLambda$ with
underlying $\LZ$-module $\Z[\zeta_n]$, there exists a
natural number $k\in\LN$ such that the convex subsets of
$\varLambda$ are determined by the $X$-rays in the directions of any set $U$ of $k$ pairwise
non-parallel $\Z[\zeta_n]$-directions. Assuming an affirmative answer,
a weak relation between $n$ and $k$ was demonstrated in~\cite[Corollary 5.5]{H3}.
\end{prob}

\section*{Final remark}
For a summary of results for model sets associated with the famous
\emph{Penrose tiling} of the plane, see~\cite{bh}. These so-called
\emph{Penrose model sets} can also be seen to be algebraic Delone
sets. The algorithmic \emph{reconstruction problem} of discrete
tomography of cyclotomic model sets has been studied
in~\cite{BG2}. In~\cite{H2}, it is shown how the results for the
planar case obtained in~\cite{BG2} and the present text can be
lifted to the practically relevant case of so-called \emph{icosahedral
  model sets} in $\R^3$. For a completer overview of both
uniqueness and computational complexity results in the discrete
tomography of Delone sets with long-range order, we refer the reader
to~\cite{H}. This reference also contains results on the interactive
concept of \emph{successive determination} of finite sets by $X$-rays
and further extensions of settings and results that are beyond our
scope here; compare also~\cite{GL}, ~\cite{H0} and~\cite{H2}.

\section*{Acknowledgements}
The author is indebted to Michael Baake, Richard J. Gardner  and Peter
A. B. Pleasants for their cooperation and for useful hints on the
manuscript. Valuable discussions with Uwe 
Grimm, Peter Gritzmann and Barbara Langfeld are gratefully acknowledged.


%

\end{document}